\documentclass[11pt,reqno]{amsart}

\usepackage{graphicx,enumerate, stmaryrd, color, url}
\usepackage{amsmath}
\usepackage{amssymb}
\usepackage{amsthm}
\usepackage{amsfonts}
\usepackage{epsfig,amscd,amsthm,mathtools,fourier}
\usepackage{tikz-cd}
\usepackage{tikz}
\usepackage{fullpage}
\usepackage[all]{xy}
\usepackage{comment}
\usepackage{times}
\usepackage{bm}
\usepackage[hidelinks]{hyperref}
\usepackage{appendix}
\usepackage{float}
\usepackage[margin=1.3in,marginparwidth=1.1in,marginparsep=0.1in]{geometry}
\usepackage{marginnote}
\usepackage{import}
\usepackage[normalem]{ulem}

\newtheorem{thm}{Theorem}[section]

\newtheorem{lma}[thm]{Lemma}
\newtheorem{prop}[thm]{Proposition}

\newtheorem{cor}[thm]{Corollary}

\theoremstyle{definition}
\newtheorem{defn}[thm]{Definition}
\newtheorem{rmk}[thm]{Remark}

\newtheorem{ex}[thm]{Example}

\newcommand{\lea}{\# l}

\newcommand{\md}{\underline{d}}

\newcommand{\mdeg}{\underline{\operatorname{deg}}}

\newcommand{\val}{\operatorname{val}}

\newcommand{\Cliff}{\mathrm{Cliff}}

\newcommand{\cO}{\mathcal{O}}

\newcommand{\Pic}{{\operatorname{{Pic}}}}

\newcommand{\stab}{\mathrm{stab}}

\newcommand{\Dh}{\mathrm{Dh}}
\newcommand{\Br}{\mathrm{Br}}

\newcommand{\GG}{\mathbb{G}}
\newcommand{\RR}{\mathbb{R}}

\newcommand{\ud}{\underline{d}}

\title{A Clifford inequality for semistable curves} 
\author{Karl Christ}

\thanks{2020 Mathematics Subject Classification: 14H51, 14H40, 14H20.\\
The author was supported by the Israel Science Foundation (grant No. 821/16) and by the Center for Advanced Studies at BGU}

\address[Christ]{$~^1$Department of Mathematics\\
	Ben-Gurion University of the Negev\\P.O.Box 653 \\Be'er Sheva\\ 84105\\  Israel and $~^2$Institute of Algebraic Geometry\\
	Leibniz University Hannover\\Welfengarten 1 \\30167 Ha\-no\-ver\\  Germany }\email{kchrist@math.uni-hannover.de}

\begin{document}
	
\maketitle

\begin{abstract}
    Let $X$ be a semistable curve and $L$ a line bundle whose multidegree is uniform, i.e., in the range between those of the structure sheaf and the dualizing sheaf of $X$. We establish an upper bound for $h^0(X,L)$, which generalizes the classic Clifford inequality for smooth curves. The bound depends on the total degree of $L$ and connectivity properties of the dual graph of $X$. It is sharp, in the sense that on any semistable curve there exist line bundles with uniform multidegree that achieve the bound. 
\end{abstract}

\section{Introduction}

For a smooth curve $X$ of genus $g$, the classic Clifford inequality states that 
\begin{equation}\label{eq:Clifford}
    h^0(X, L) \leq \frac{d}{2} + 1,
\end{equation}
for any line bundle $L$ with degree $d$ in the special range $0 \leq d \leq 2g - 2$. It is the first part of Clifford's Theorem, and the second part states that this bound is achieved only by the structure sheaf $\mathcal O_X$, the canonical sheaf $\omega_X$, and multiples of the $g_2^1$ if $X$ is hyperelliptic.

In this paper, we are interested in generalizations of the Clifford inequality \eqref{eq:Clifford} to semistable curves. Recall that $X$ is semistable if it is reduced with nodal singularities, and the degree of the dualizing sheaf $\omega_X$ is nonnegative on each irreducible component of $X$.
Semistable curves, or rather the more restrictive class of stable curves, are the most commonly studied singular curves, since they give a well-understood compactification of the moduli space of smooth curves \cite{DM}. They provide important tools for understanding the geometry of the moduli space, as well as studying smooth curves via degeneration techniques.

Eisenbud, Koh and Stillman \cite{EKS} showed in an appendix with Harris, that the Clifford inequality \eqref{eq:Clifford} still holds if $X$ is reduced and irreducible, but not necessarily smooth. On the other hand, if $X$ is reducible, an upper bound on $h^0(X, L)$ purely in terms of the total degree $d$ is impossible. Indeed, by setting the degree of $L$ on one irreducible component very negative, one obtains an arbitrarily large degree of $L$ on another component, while maintaining the total degree $d$. Arbitrarily large values of $h^0(X, L)$ can be realized in this way for any fixed $d$.

Thus the issue lies with the distribution of the total degree $d$ of $L$ among the irreducible components of $X$. We collect the degrees $\deg(L|_{X_v})$ of $L$ on irreducible components $X_v$ in a tuple of integers $\md = \mdeg(L)$, the \emph{multidegree} of $L$.

We restrict the multidegrees that we allow for to \emph{uniform} ones. They generalize the numerical condition $0 \leq d \leq 2g - 2$ to multidegrees. That is, $L$ has uniform multidegree if \[\underline 0 \leq \mdeg(L) \leq \mdeg(\omega_X),\]
where $\underline 0$ has value $0$ on each irreducible component and $\mdeg(\omega_X)$ is the multidegree of the dualizing sheaf $\omega_X$ on $X$.

\subsection{Results}

To state our main result, we need one more combinatorial invariant. Recall that to any nodal curve $X$ we can associate its dual graph $\GG_X$. The Clifford inequality for uniform multidegrees depends on certain connectivity properties of $\GG_X$. 

More precisely, we denote by $\GG_X^\Br$ the graph obtained from $\GG_X$ by contracting all edges of $\GG_X$ that are not bridges. That is, the edges whose removal does not increase the number of connected components of $\GG_X$. Thus $\GG_X^\Br$ contains an edge for each separating node of $X$ and a vertex for each connected component of the partial normalization of $X$ at those separating nodes. If $X$ is connected, $\GG_X^\Br$ is a tree since each of its edges is a bridge, and we call it the \emph{tree of $2$-edge-connected components}. In this case, we denote by $\lea\left(\GG_X^\Br\right)$ the number of leaves of $\GG_X^\Br$, that is, the number of vertices adjacent to a single edge. In case $\GG_X^\Br$ contains a single vertex, we set $\lea\left(\GG_X^\Br\right) = 2$. If $X$ is not connected, we define $\lea\left(\GG_X^\Br\right)$ by summing the values defined as above for each connected component. See Subsection~\ref{subsec: tree of components} for further details.

\begin{thm}[Clifford inequality for uniform multidegrees]\label{thm:main}
Let $X$ be a semistable curve and $L$ a line bundle on $X$ with uniform multidegree and of total degree $d$. Then \begin{equation} \label{eq:uniform bound} h^0(X,L) \leq \frac{d}{2} + \frac{\lea\left(\GG_X^\Br\right)}{2}.\end{equation}
\end{thm}

See Theorem~\ref{thm:clifford for uniform multidegrees}. 
If $X$ is connected and contains no separating nodes, then by definition $\lea\left(\GG_X^\Br\right) = 2$ and hence we immediately obtain the classic Clifford inequality in this case:

\begin{cor}
Let $X$ be a connected, semistable curve without separating nodes and $L$ a line bundle with uniform multidegree of total degree $d$. Then $L$ satisfies the classic Clifford inequality. 
\end{cor}

In Proposition~\ref{prop:vertex of valence three} and Proposition~\ref{prop:generic case}, we show that the bound of Theorem~\ref{thm:main} is sharp, but not achieved generically:

\begin{thm} \label{thm:main2}
On any semistable curve $X$ there exist line bundles $L$ with uniform multidegree such that \[h^0(X,L) = \frac{\deg(L)}{2} + \frac{\lea\left(\GG_X^\Br\right)}{2}.\] A general line bundle $L$ of fixed uniform multidegree satisfies the classic Clifford inequality.
\end{thm}

Instead of uniform multidegrees, there are other choices for classes of multidegrees one can consider. One of the most important ones is the class of semistable (or more restrictively, stable) multidegrees, used in the construction of the universal compactified Jacobian~\cite{Caporasocompactification}. In general, there are uniform multidegrees that are not semistable and \emph{vice versa}. Furthermore, the class of semistable multidegrees have weaker properties with respect to an upper bound on $h^0(X,L)$. That is, in general, they do not satisfy the inequality in Theorem~\ref{thm:clifford for uniform multidegrees}. See Example~\ref{ex:semistable clifford}. We show however in Lemma~\ref{lma:stable in degree g-1}, that every stable multidegree of total degree $g - 1$ is uniform and obtain:
\begin{cor}
Suppose $X$ is a connected, stable curve without separating nodes. Then every line bundle $L$ of degree $g - 1$ with stable multidegree satisfies the classic Clifford inequality. 
\end{cor}
	
\subsection{Previous results}

The study of limits of line bundles and their sections has a long history in algebraic geometry. The problem is known to be difficult, as shown by the various approaches that have been developed to describe it. Prominent ones have been limit linear series, compactified Jacobians, tropical and logarithmic divisors, and enriched structures. Each approach comes with its own advantages and limitations, and in each case there is a large body of literature. See, for example, \cite{EH86} \cite{O19}, \cite{OdaSeshadri79} \cite{Caporasocompactification}, \cite{Baker} \cite{MW22} and \cite{M98} \cite{AMIII}.

More specific to our question of an upper bound on the rank in terms of the degree on singular curves, we already mentioned the generalization of Clifford's theorem to integral curves \cite{EKS}. The closely related question of the Clifford index for singular curves has been studied in \cite{BE91} and \cite{Franciosi}. Both papers use connectivity properties of the dual graph. Furthermore, the edge-connectivity of the dual graph is known to determine positivity properties of the dualizing sheaf \cite{Catanese} \cite{CFHR}. The invariant $\lea(\GG_X^\Br)$ however seems to be new in this context. 

Special cases of Theorem~\ref{thm:clifford for uniform multidegrees} are covered by results of Franciosi and Tenni \cite{FT14} (see also \cite{Franciosi}). Namely, they show that the classic Clifford inequality holds for uniform multidegrees if $X$ has no separating nodes, and either not all global sections of the residual $\omega_X \otimes L^{-1}$ of $L$ vanish on any given irreducible component of $X$, or $X$ is $4$-connected \cite[Theorems A and 3.9]{FT14}. Without the last assumption, but still in the case with no separating nodes, they obtain a weaker bound in \cite[Theorem 3.8]{FT14}.

Caporaso's study \cite{caporasosemistable} of the situation for semistable multidegrees establishes that the classic Clifford inequality holds for semistable multidegrees in the following cases: if $X$ has $2$ irreducible components; if the total degree is $0$ or $2g - 2$; and if $X$ has no separating nodes and the total degree is at most $4$ \cite[Theorem 3.3; Theorems 4.2 and 4.4; Theorem 4.11]{caporasosemistable}. 

On the other hand, in \cite[Theorem 3.14]{Franciosi} it is claimed that every line bundle of uniform multidegree satisfies the classic Clifford inequality, and a similar claim is made in \cite[Proposition 3.1]{caporasosemistable}. Theorem~\ref{thm:main2} shows that this assertion is somewhat too optimistic, at least if $X$ contains separating nodes. See Remark~\ref{rmk:other statements clifford} for more details.

Finally, Caporaso, Len and Melo \cite{CLM} use the Baker-Norine rank \cite{BN} of the multidegrees, to show that any multidegree is equivalent via chip--firing to a multidegree for which all line bundles satisfy the classic Clifford inequality \eqref{eq:Clifford}. It remains an intriguing open problem to characterize the class of such multidegrees, or even to find explicit representatives in each chip-firing class. 

\subsection{Structure of the paper}

In Sections~\ref{sec:conventions} and \ref{sec: preliminaries} we fix notations and recall some background. In Section~\ref{sec:graph theory} we collect the graph-theoretical definitions and observations that will be needed in the statement and proof of Theorem~\ref{thm:main}. More precisely, we introduce uniform multidegrees in Subsection~\ref{subsec: uniform}, discuss the tree of $2$-edge-connected components $\GG_X^\Br$ in Subsection~\ref{subsec: tree of components}, and Dhar subgraphs in Subsection~\ref{subsec:Dhar}. In Section~\ref{sec:counterexamples} we describe counterexamples to the classic Clifford inequality in case of uniform multidegrees. Proposition~\ref{prop:vertex of valence three} establishes that the bound in Theorem~\ref{thm:main} is sharp. In Subsection~\ref{subsec:clifford inequality} we prove the Clifford inequality for uniform multidegrees, and in Subsection~\ref{subsec: generic} we establish the classic Clifford inequality for a general line bundle of fixed uniform multidegree. Finally, in Subsection~\ref{subsec:semistable multidegrees} we discuss the relationship with stable multidegrees. 
\medskip
	
\noindent 
{\bf Acknowledgements.} Many discussions with Lucia Caporaso and Ilya Tyomkin helped shape this paper, and I am very grateful for the insights and suggestions they provided. In addition, I would like to thank Matt Baker, Sam Payne, Sara Torelli and an anonymous referee, whose comments on an earlier draft led to many improvements in presentation. 

\section{Notation and Conventions} \label{sec:conventions}

Throughout the paper, we work over an algebraically closed field $k$ of characteristic $0$. We consider curves $X$ over $k$, which we will always assume to be reduced with nodal singularities. 

We denote the \emph{dual graph} of $X$ by $\GG_X$. That is, $\GG_X$ contains a vertex $v$ for every irreducible component $X_v$ of $X$; an edge between vertices $v$ and $w$ for each node in $X_v \cap X_w$, possibly with $v = w$; and each vertex $v$ is assigned the weight $g_v$ given by the geometric genus of $X_v$. In particular, $\GG_X$ may contain multiple edges between the same two vertices, as well as loop edges. We denote by $V(\GG_X)$ and $E(\GG_X)$ the sets of vertices and edges of $\GG_X$, respectively.

An edge of $\GG_X$ is a \emph{bridge} if the graph obtained by removing the edge has more connected components than $\GG_X$. A node of $X$ corresponding to a bridge of $\GG_X$ is called a \emph{separating node}.

The \emph{genus} of $\GG_X$ is defined as 
\[
g(\GG_X) = 1 - \chi(\GG_X) + \sum_{v \in V(\GG_X)} g_v,
\]
where $\chi(\GG_X)$ is the Euler characteristic of $\GG_X$. That is, $1 - \chi(\GG_X) = 1 - |V(\GG_X)| + |E(\GG_X)|$.  The genus of $\GG_X$ equals the arithmetic genus $g(X)$ of $X$. We write $g \vcentcolon= g(X) = g(\GG_X)$ if $X$ is clear from the context.

For a subcurve $Y \subset X$ we write $Y^c = \overline{X \setminus Y}$
for the closure of the complement in $X$. Any subcurve $Y \subset X$ corresponds to an \emph{induced} subgraph $\GG_Y$ of $\GG_X$, that is, one that contains all edges of $\GG_X$ between vertices contained in $\GG_Y$. The edges adjacent to $\GG_Y$ but not contained in it, correspond to the nodes in $Y \cap Y^c$. 

We denote by $\val(v)$ the \emph{valence} of $v \in \GG_X$; that is, the number of edges adjacent to $v$, with loops counted twice. We denote by $\omega_X$ the \emph{dualizing sheaf} of $X$. It has total degree $2g - 2$. The restriction of $\omega_X$ to an irreducible component $X_v$ is given as \[\left(\omega_X\right)|_{X_v} \simeq \omega_{X_v}\left(X_v \cap X_v^c\right),\] where $\omega_{X_v}$ is the dualizing sheaf of $X_v$. In particular, $\omega_X$ has degree $2g_v - 2 + \val(v)$ on $X_v$.

A curve $X$ is \emph{semistable}, if $\val(v) \geq 2$ whenever $X_v$ is rational. It is \emph{stable} if $\val(v) \geq 3$, whenever $X_v$ is rational and $\val(v) \geq 1$ whenever $g_v = 1$. Equivalently, $X$ is semistable (stable) if $\omega_X$ has non-negative (positive) degree on each irreducible component of $X$.  

We write $\md$ for a \emph{multidegree}, that is, a formal linear combination of vertices of $\GG_X$ with integer coefficients. We denote by $\md_v$ the coefficient at a vertex $v$ of $\GG_X$. The multidegree $\mdeg(L)$ of a line bundle $L$ is defined to have value $\deg(L|_{X_v})$ on $v$. The \emph{total degree} of $\md$ is $\sum_v \md_v$, which coincides with the total degree $\deg(L)$ of $L$ if $\md$ is the multidegree $\mdeg(L)$ of $L$. 

The \emph{Picard scheme} of $X$ is denoted by $\Pic(X)$. The sublocus parametrizing line bundles of degree $d$ is denoted by $\Pic^d(X)$. The connected components of $\Pic^d(X)$ are denoted by $\Pic^{\md}(X)$, one for each multidegree $\md$ of total degree $d$. $\Pic^{\md}(X)$ parametrizes line bundles of multidegree $\md$ and any two such connected components are isomorphic. If $X$ is irreducible, there is a unique connected component of $\Pic^d(X)$; otherwise, there are infinitely many. 

Clearly, $\Pic^{\md}(X) \simeq \Pic^{\md_1}(X_1) \times  \Pic^{\md_2}(X_2)$ if $X_1$ and $X_2$ are two connected components of $X$ and $\md_i$ denotes the restriction of $\md$ to the corresponding connected component. 
For a connected curve $X$ we have the following short exact sequence \begin{equation} \label{eq:ses_pic}
    1 \rightarrow (k^*)^{1 - \chi(\GG_X)} \rightarrow \Pic^{\md}(X) \rightarrow \Pic^{\md} \left (X^\nu \right ) \rightarrow 0,
\end{equation}
where $X^\nu$ denotes the normalization of $X$. The second map is given by pulling back the line bundle. For the definition of the first map see for example \cite[Theorem 2.3]{Alexeev}. Roughly speaking, a line bundle on $X$ is specified by its pull-back to the normalization together with gluing data over the nodes.

Given a node $p \in X$ we still have a surjective morphism $\Pic^{\md}(X) \to \Pic^{\md}(X^\nu_p)$ given by pull-back, where $X_p^\nu$ denotes the partial normalization of $X$ at $p$. It is an isomorphism if $p$ is a separating node and otherwise has fiber $k^*$. In particular, if $p$ is a separating node, then $\Pic(X) \simeq \Pic(X_1) \times \Pic(X_2)$, where the $X_i$ denote the connected components of $X_p^\nu$.

\section{Global sections of line bundles on nodal curves} \label{sec: preliminaries}

In this section, we continue the preliminaries and collect some well-known results about the dimension $h^0(X,L)$ of the space of global sections of a line bundle $L$ on $X$. 

\medskip

We begin with the Clifford Theorem in case $X$ is irreducible. We cite only the parts that are used later and the statement found in \cite[Theorem A, p. 533]{EKS} is stronger. 

\begin{thm}[Clifford inequality] \label{thm:clifford irreducible}
Suppose $X$ is an irreducible nodal curve and $L$ a line bundle of degree $0 \leq d \leq 2g - 2$. Then \[h^0(X,L) \leq \frac{d}{2} + 1.\] Furthermore, there is a dense open subset of $\Pic^d(X)$ in which the inequality is strict.
\end{thm}

Next, recall that the Riemann-Roch Theorem and Serre duality still hold for nodal curves. See for example \cite[pp. 90 - 91]{curvesv2}.

\begin{thm}[Riemann-Roch]\label{thm:Riemann Roch}
Let $X$ be a nodal curve of genus $g$ and $L$ a line bundle on $X$ of total degree $d$. Then \[h^0(X, L) - h^0\left(X, \omega_X \otimes L^{-1}\right) = d - g + 1.\]
\end{thm}

\begin{cor}\label{cor:RR applied1}
Let $X$ be a nodal curve, $p \in X$ a smooth point and $L$ a line bundle on $X$. Then $p$ is a base point of $L$ if and only if it is not a base point of $(\omega_X \otimes L^{-1})(p)$. 
\end{cor}

\begin{proof}
Applying the Riemann-Roch Theorem gives on the one hand \[h^0\left(X,L\right) - h^0\left(X, \omega_{X} \otimes L^{-1}\right) = d - g + 1\] and on the other \[h^0\left(X,L(-p)\right) - h^0\left(X, (\omega_{X} \otimes L^{-1})(p)\right) = d - 1 - g + 1.\]
Hence $h^0(X, L) = h^0(X, L(-p))$ if and only if $h^0(X, \omega_{X} \otimes L^{-1}) = h^0\left(X, (\omega_{X} \otimes L^{-1})\right)(p) - 1$.
\end{proof}

\begin{cor}\label{cor:RR applied2}
Let $X$ be a nodal curve, $p \in X$ a smooth point. Then $p$ is a base point of $\omega_X(p)$.
\end{cor}	

\begin{proof}
The Riemann-Roch Theorem gives $h^0\left(X, \omega_X\right) = h^0\left(X, \omega_X(p)\right) = g$.
\end{proof}

As a last immediate consequence of the Riemann-Roch Theorem, we observe that any Clifford type inequality we want to prove holds for $L$ if and only if it holds for its residual $\omega_X \otimes L^{-1}$.

\begin{cor} \label{cor: clifford residual}
Let $X$ be a semistable curve, $L$ a line bundle of total degree $d$ on $X$ and $l \in \RR$ any constant. Then \[h^0(X,L) \leq \frac{d}{2} + l \Leftrightarrow h^0(X,\omega_X \otimes L^{-1}) \leq \frac{2g - 2 - d}{2} + l.\]
\end{cor}

\begin{proof}
    By the Riemann-Roch Theorem  we have
    \begin{eqnarray*}
    & & h^0(X,L) \leq \frac{d}{2} + l \\
    &\Leftrightarrow& h^0(X,\omega_X \otimes L^{-1}) + d - g + 1 \leq \frac{d}{2} + l\\
    &\Leftrightarrow& h^0(X,\omega_X \otimes L^{-1})\leq \frac{2g - 2 - d}{2} + l.
    \end{eqnarray*}
\end{proof}

\medskip

Finally, we give two lemmas that help to calculate $h^0(X,L)$ inductively.
\begin{defn} \label{def: neutral pair}
   A \emph{neutral pair} of $L$ is a pair of smooth points $p_1,p_2$ on $X$, such that \[h^0\left(X, L(-p_1)\right) = h^0\left(X, L(-p_2)\right) = h^0\left(X, L(-p_1 - p_2)\right).\]
\end{defn} 
Notice that when $p_1$ and $p_2$ are contained in different connected components of $X$, they are a neutral pair if and only if they are both base points of $L$.

Next, denote by $\nu: X^\nu \to X$ the partial normalization of $X$ at a single node $p$. We have for any line bundle $L$ on $X$ \[h^0\left(X^\nu, \nu^* L\right) - 1 \leq h^0\left(X,L\right) \leq h^0\left(X^\nu, \nu^* L\right).\] The following lemma describes the possible cases in detail, see \cite[Lemma 1.4]{caporasosemistable}.

\begin{lma} \label{lma:neutral pairs}
Let $p$ be a node of a nodal curve $X$. Let $\nu: X^\nu \to X$ be the partial normalization of $X$ at $p$ and $p_1, p_2$ the two preimages of $p$. Let $L^\nu$ be a line bundle on $X^\nu$ with $h^0(X^\nu, L^\nu) \neq 0$.

Then there exists $L$ on $X$ with $L^\nu = \nu^* L$ and $h^0(X^\nu, L^\nu) = h^0(X, L)$ if and only if $p_1$ and $p_2$ are a neutral pair of $L^\nu$. If $p_1$ and $p_2$ are not base points, there is at most one such $L$. 
\end{lma}

Recall that we denote by $Y^c = \overline{X \setminus Y}$ the closure of the complement of a subcurve $Y$. In particular, $Y^c \cap Y$ is a finite union of nodes of $X$, that are smooth points both on $Y$ and $Y^c$.

\begin{lma}\label{lma:subcurves}
Let $X$ be a nodal curve, $L$ a line bundle on $X$ and $Y \subsetneq X$ a proper subcurve. Then \[h^0(X,L) \leq h^0\left(Y, L|_Y\right) + h^0\left(Y^c, L|_{Y^c}\left(- Y \cap Y^c\right)\right).\] 
If all global sections of $L$ vanish along $Y$, then $h^0(X,L) = h^0\left(Y^c, L|_{Y^c}\left(- Y \cap Y^c\right)\right)$.
\end{lma}

\begin{proof}
The restriction map $H^0(X,L) \to H^0(Y, L|_Y)$ is linear with kernel the global sections of $L$, that vanish along $Y$. Hence we may naturally identify this kernel with $H^0\left(Y^c, L|_{Y^c}\left(- Y \cap Y^c\right)\right)$ and both claims follow immediately. 
\end{proof}

\section{Graph theoretic notions} \label{sec:graph theory}

In this section, we introduce three notions of a combinatorial flavour: uniform multidegrees, the tree/forest of $2$-edge connected components and Dhar subgraphs. In each case, we collect some observations that will be used in the proof of Theorem~\ref{thm:clifford for uniform multidegrees}.

\subsection{Uniform multidegrees}\label{subsec: uniform}
Recall that $\val(v)$ denotes the number of edges adjacent to a vertex $v$ in $\GG_X$, with loops counted twice, and that the dualizing sheaf $\omega_X$ has degree $2g_v - 2 + \val(v)$ on the irreducible component $X_v$. 

\begin{defn} \label{def:uniform}
Let $X$ be a semistable curve and $L$ a line bundle on $X$ of multidegree $\md$. We call the multidegree $\md$ \emph{uniform} if it satisfies on each irreducible component $X_v$ of $X$: \[0 \leq \md_v \leq 2 g_v - 2 + \val(v) = \deg\left(\omega_X|_{X_v}\right).\]
\end{defn}

\tikzset{every picture/.style={line width=0.75pt}}
    
    \begin{figure}[ht]
    \begin{tikzpicture}[x=0.8pt,y=0.8pt,yscale=-0.8,xscale=0.8]
    \import{./}{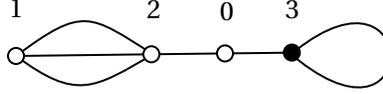}
    \end{tikzpicture}	
    \caption{
       An example of $\mdeg(\omega_X)$ on a dual graph $\GG_X$. The bold vertex has weight $1$ and vertices drawn as circles weight $0$. 
    } \label{figure_clifford1}
\end{figure}

\begin{rmk} \label{rmk:uniform residual}
A multidegree $\mdeg(L)$ is uniform, if and only if its residual multidegree $\mdeg(\omega_X \otimes L^{-1})$ is uniform. If a multidegree $\md$ is uniform of total degree $d$, then $0 \leq d \leq 2g - 2$, while the converse is not true. Other than in the irreducible case, if $L$ does not have uniform multidegree, it still can be special -- that is, $h^0(X,L) > 0$ and $h^0(X, \omega_X \otimes L^{-1}) > 0$. See \cite{chr1} for a detailed discussion.
\end{rmk}

\begin{rmk}\label{rmk:uniform nodal}
If $X$ is nodal but not semistable, it contains a smooth rational component attached at a single node. On this component, the dualizing sheaf has negative degree, and hence there exist no uniform multidegrees on such a curve.
\end{rmk}  

\medskip

Recall that the \emph{stabilization} $X_\stab$ of a connected semistable curve $X$ of genus at least $2$ is given by contracting all rational components $X_v$ with $\val(v) = 2$. Such components are called \emph{exceptional components}.

\begin{lma}\label{lma:semistable degree 0}
Let $X$ be a connected semistable curve of genus at least $2$, $\pi:X \to X_\stab$ its stabilization and $\md$ a uniform multidegree. Then $\pi_*$ induces an isomorphism $\Pic^{\md}(X) \to \Pic^{\md'}(X_\stab)$ that satisfies $h^0(X, L) = h^0(X_\stab, \pi_* L)$.
\end{lma}

\begin{proof}
Let $X_v$ be a rational component of $X$ with $\val(v) = 2$. Hence $2g_v - 2 + \val(v) = 0$ and since $L$ has uniform multidegree, $\deg(L|_{X_v}) = 0$. This ensures that $\pi_* L$ indeed is a line bundle on $X_\stab$, see for example \cite[Theorem 3.1 (3)]{EstevesPacini16}. In \emph{loc. cit} it is furthermore shown that in this case $\pi^* \pi_* L \simeq L$, and thus the map $\pi^* \colon \Pic^{\md'}(X_\stab) \to \Pic^{\md}(X)$ induced by $\pi^*$ gives an inverse to the one induced by $\pi_*$.  Finally, the last claim, $H^0(X, L) \simeq H^0(X_\stab, \pi_* L)$, holds by definition of the push-forward.
\end{proof}

Recall that the \emph{Clifford index} of a smooth curve $X$ is defined as \[\Cliff(X) \coloneqq \min_{L \in \Pic(X)} \left\{ \deg(L) - 2 h^0(X,L) + 2 \mid h^0(X,L) \geq 2 \text{ and } h^1(X,L) \geq 2 \right\}.\]  If $X$ is reducible, we can define the Clifford index of $X$ as the minimum as above, requiring of $L$ in addition to have uniform multidegree. This is a variation of the definition introduced by Franciosi \cite[Definition 3.5]{Franciosi} (see \cite{BE91} for possible different definitions).

\begin{cor}\label{cor:cliffordindex}
Let $X$ be a semistable curve of genus at least $2$ and $\pi \colon X \to X_\stab$ its stabilization. Then \[\Cliff(X) = \Cliff(X_\stab).\]
\end{cor}

\begin{proof}
    In the proof of Lemma~\ref{lma:semistable degree 0} we saw that if $L$ has uniform multidegree $\md$, then it has value $0$ on exceptional components and hence $\pi_* L$ is a line bundle. We write $\pi_* \md$ for the multidegree of $\pi_* L$ in this case. 
    The claim follows immediately from Lemma~\ref{lma:semistable degree 0}, once we establish that $\pi_*$ induces a bijection on the set of uniform multidegrees on $X$ and $X_\stab$.
    
    On the one hand, $\pi_* \md$ has value $\md_v$ on any vertex $v$ of $\GG_{X_\stab}$. The valence of $v$ in $\GG_{X_\stab}$ is equal to the valence of $v$ viewed as a vertex in $\GG_X$, since $X$ is semistable. Thus if $\md$ is uniform, then so is $\pi_* \md$. Furthermore, $\pi_*$ is injective on the set of uniform multidegrees since every uniform multidegree on $X$ has value $0$ on exceptional components.
    
    On the other hand, if $\md$ is a uniform multidegree on $X_\stab$, the multidegree on $X$ obtained by setting all values on exceptional components equal to $0$ gives a uniform multidegree whose image under $\pi_*$ is $\md$, so $\pi_*$ is also surjective on the set of all uniform multidegrees on $X_\stab$.
\end{proof}

\subsection{The tree of 2-edge connected components.}
\label{subsec: tree of components}
Recall that an edge $e$ of $G$ is called a bridge, if $G - e$ has more connected components than $G$. Here $G - e$ denotes the graph obtained from $G$ by removing $e$. A graph $G$ is called a \emph{forest} if every edge is a bridge. It is called a \emph{tree}, if $G$ is in addition connected. A \emph{leaf} in a forest or tree is a vertex that is adjacent to a single edge. 

\begin{defn} \label{def:number of leaves}
Let $G$ be a tree. We denote its \emph{number of leaves} by $\lea(G)$, where we set $\lea(G) = 2$ if $G$ contains a single vertex.
If $G$ is more generally a forest, we set $\lea(G) = \sum_i \lea(G_i)$, where the $G_i$ are the connected components of $G$. 
\end{defn}

A graph $G$ is called \emph{$2$-edge-connected}, if it is connected and contains no bridges. The \emph{$2$-edge-connected components} of a graph $G$ are the maximal $2$-edge-connected subgraphs.

\begin{defn}\label{def:G2}
Let $G$ be a graph. The associated \emph{forest of $2$-edge-connected components} $G^\Br$ is the graph obtained from $G$ by contracting all edges of $G$ that are not bridges. If $G$ is connected, we will call $G^\Br$ the \emph{tree of $2$-edge-connected components}.
\end{defn}

The edges of $G^\Br$ correspond bijectively to bridges of $G$ and the vertices of $G^\Br$ to $2$-edge-connected components of $G$.
Furthermore, as the name suggests, $G^\Br$ is in general a forest, and a tree if $G$ is connected.

\tikzset{every picture/.style={line width=0.75pt}}
    
    \begin{figure}[ht]
    \begin{tikzpicture}[x=0.8pt,y=0.8pt,yscale=-0.8,xscale=0.8]
    \import{./}{figure_clifford2.tex}
    \end{tikzpicture}	
    \caption{
       An example of $G^\Br$ with $\lea\left(G^\Br\right) = 3$. 
    } \label{figure_clifford2}
\end{figure}

Recall that a subgraph $H$ of $G$ is called an induced subgraph, if $H$ contains all edges of $G$ between vertices that are contained in $H$. 

\begin{lma} \label{lma:splitting of number of leaves}
Let $G$ be a connected graph and $H \subsetneq G$ an induced subgraph. Denote by $k$ the number of edges of $G$ that are adjacent to $H$ but not contained in $H$. Then
\begin{enumerate}
    \item We have \[\lea\left(H^\Br\right) \leq \lea\left(G^\Br\right) + k.\]
    \item If $G$ is $2$-edge-connected, then \[\lea\left(H^\Br\right) \leq k.\]
\end{enumerate}
\end{lma}

\begin{proof}
    It suffices to show the claim for $H$ connected, the general case then follows by applying the claim to each connected component. So assume $H$ is connected.
    
    Assume first $H$ contains no bridges, i.e., it is $2$-edge-connected. Then by definition $\lea \left(H^\Br \right) = 2$. Since $\lea\left(G^\Br \right) \geq 2$, the first claim is immediate. If $G$ is $2$-edge-connected, we have $k \geq 2$. Otherwise there would be a single edge connecting $H$ to its complement in $G$, which then is a bridge. Thus also the second claim follows in this case.
    
    Now let $e \in E(H)$ be a bridge of $H$, whose image in $H^\Br$ is the unique edge adjacent to a leaf $v_i$ of $H^\Br$. Denote by $H_i \subset H$ the $2$-edge connected component corresponding to $v_i$. Assume furthermore, that $e$ is not a bridge in $G$ and thus $v_i$ not a leaf in $G^\Br$. Hence there has to be an edge $e_i$ adjacent to $H_i$ that is not contained in $H$. Since we obtain for each leaf $v_i$ of $H^\Br$ that is not a leaf of $G^\Br$ a different $e_i$, there are at most as many leaves as there are edges of $G$ adjacent to $H$, but not contained in it. This gives both claims in this case.
\end{proof}

\begin{lma}\label{lma: removing edge}
Let $G$ be a $2$-edge connected graph. Then for any edge $e$ of $G$ we have \[\lea\left((G -e)^\Br\right) = \lea\left(G^\Br\right) = 2.\]
\end{lma}

\begin{proof}
    Since $G$ is $2$-edge connected, $G^\Br$ consists of a single vertex and hence $\lea\left(G^\Br\right) = 2$ by definition. If $G - e$ remains $2$-edge connected, the claim is immediate. Otherwise, let $G_i \subsetneq G - e$ be a $2$-edge connected component, that is, a maximal subgraph that is $2$-edge connected. Assume furthermore, that $G_i$ corresponds to a leaf of $(G-e)^\Br$, that is, it is adjacent to a single bridge of $G - e$. Since $G$ is $2$-edge connected its unique $2$-edge connected component is $G$ itself. In particular, $G_i$ is not a $2$-edge connected component of $G$ since $G_i \neq G$. It follows that $G_i$ needs to be adjacent to an edge of $G$ that is not contained in $G - e$. Since $e$ is the only such edge, and $e$ can be adjacent to only two distinct subgraphs $G_i$ as above, it follows that $(G-e)^\Br$ has two leaves and thus $\lea\left((G -e)^\Br\right) = 2$, as claimed.
\end{proof}

As for curves, to any semistable graph $G$ of genus at least $2$, we can associate a unique stable graph $G_\stab$ by successively contracting edges that are adjacent to $2$-valent vertices of weight $0$. We call $G_\stab$ the stabilization of $G$. If $G$ is the dual graph of a semistable curve $X$, then $G_\stab$ is the dual graph of the stabilization $X_\stab$ of $X$.

\begin{lma}\label{lma:leaves_stabilization}
Let $G$ be a semistable graph of genus at least $2$ and $G_\stab$ its stabilization. Then \[\lea(G^\Br) = \lea\left(G_\stab^\Br\right).\]
\end{lma}

\begin{proof}
    A choice of edge contraction map $\pi \colon G \to G_\stab$ allows to realize the edges of $G_\stab$ as a subset of the edges of $G$.
    It is easy to see, that the bridges of $G_\stab$ are bridges of $G$ under this identification and the additional bridges of $G$ all get contracted by $\pi$ \cite[Lemma 2.1.1]{CC}.
    In particular, we have a well-defined map $\pi^\Br \colon G^\Br \to G_\stab^\Br$ given by contracting bridges of $G$ that get contracted by $\pi$.
    
    The map $\pi$ can be realized by successively contracting single edges. Thus we may assume that $\pi$ contracts a single edge $e$ of $G$. Since $e$ gets contracted by $\pi$ and $G$ is semistable, it is adjacent to a $2$-valent weight $0$ vertex $v$. By what we said above, $\pi^\Br$ is trivial if $e$ is not a bridge of $G$. If $e$ is a bridge, then also the second edge adjacent to $v$ needs to be a bridge. Thus $v$ is itself a $2$-edge connected component of $G$ and we may view it as a vertex of $G^\Br$. Then $\pi^\Br$ contracts $e$, viewed as an edge of $G^\Br$ and $G^\Br$ is obtained from $G_\stab^\Br$ by replacing one edge with a weight $0$ vertex and two adjacent edges. In particular, the number of leaves of the two graphs remains the same.   
\end{proof}

\subsection{Dhar subgraphs} \label{subsec:Dhar}
Let $G$ be a graph, $v \in V(G)$ a vertex and $\md$ a multidegree on $G$. Following \cite[\S 3.4]{CLM}, we define a sequence of induced subgraphs \begin{equation} \label{eq:Dhar decomposition}
    H_0 = \{v\} \subset H_1 \subset \ldots \subset H_n = \Dh(v, \md),
\end{equation}
the Dhar decomposition, iteratively as follows. Given a subgraph $H_i$, consider the multidegree $\md'$ obtained from $\md$ by subtracting from each vertex $v \in V(G)\setminus V(H_i)$ the number of edges that are adjacent to $v$ and $H_i$. Then $H_{i+1}$ is the induced subgraph of vertices in $H_i$ and those vertices adjacent to $H_i$, on which $\md'$ is negative. Since there are only finitely many vertices, at some point $H_{n} = H_{n+1}$ and we define $\Dh(v, \md) \coloneqq H_{n}$ for the \emph{Dhar subgraph associated to $\md$ and $v$}.

\tikzset{every picture/.style={line width=0.75pt}}
    
    \begin{figure}[ht]
    \begin{tikzpicture}[x=0.8pt,y=0.8pt,yscale=-0.8,xscale=0.8]
    \import{./}{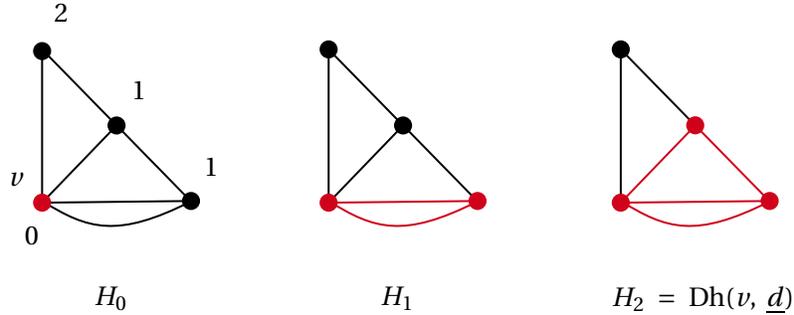}
    \end{tikzpicture}	
    \caption{
       A Dhar decomposition. The $H_i$ are the subgraphs in red.  
    } \label{figure_clifford3}
\end{figure}

\begin{lma}\label{lma:dhar restriction}
Let $L$ be a line bundle of multidegree $\md$ and $v \in V(\GG_X)$ a vertex corresponding to an irreducible component $X_v$. Suppose that the Dhar set equals all of $\GG_X$, that is, $\Dh(v, \md) = \GG_X$. Then the restriction map $H^0(X, L) \to H^0(X_v, L|_{X_v})$ is injective. In particular, 
$h^0(X,L) \leq h^0(X_v, L|_{X_v})$.
\end{lma}

\begin{proof}
Suppose a global section $s$ of $L$ is in the kernel of $H^0(X, L) \to H^0(X_v, L|_{X_v})$, that is, it vanishes on $X_v$. We need to show that then $s$ vanishes on all of $X$. 

Let \[Y_0 = X_v \subset Y_2 \subset \ldots \subset Y_n = X\] be the subcurves corresponding to the graphs $H_i$ in the Dhar decomposition \eqref{eq:Dhar decomposition}. Since $s$ vanishes on $X_v$ and $\Dh(v, \md) = \GG_X$, it suffices to show that if $s$ vanishes on $Y_i$ then it vanishes on $Y_{i + 1}$. So let $w \in V(H_{i + 1}) \setminus V(H_i)$ be a vertex of $H_{i +1}$ but not of $H_i$. By construction of the $H_i$, we then have $\md_w < |X_w \cap Y_i|$. On the other hand, if $s$ vanishes on $Y_i$, it vanishes in particular at $X_w \cap Y_i$. Since $\md_w$ is the degree of the restriction of $L$ to $X_w$ this implies that $s$ vanishes on all of $X_w$.
\end{proof}

Recall that we defined uniform multidegrees in Definition~\ref{def:uniform}.

\begin{lma}\label{lma:dhar uniform}
Let $X$ be a semistable curve, $v \in V(\GG_X)$ a vertex and $L$ a line bundle on $X$ with uniform multidegree $\md$. Let $Y$ be the subcurve of $X$ corresponding to the Dhar set $\Dh(v, \md)$, and $Y^c$ the closure of its complement. Then if $Y^c$ is not empty, it is semistable and the multidegree of $L|_{Y^c}(-Y \cap Y^c)$ is uniform on $Y^c$.
\end{lma}

\begin{proof}
    If $Y^c$ is not semistable, there exists no uniform multidegrees on it by Remark~\ref{rmk:uniform nodal}. Hence it suffices to show that $L|_{Y^c}(-Y \cap Y^c)$ has uniform multidegree. Let $H = \GG_{Y^c}$ be the subgraph corresponding to $Y^c$, that is, the induced subgraph with vertices $V(\GG_X) \setminus V(\Dh(v, \md))$. Let $\md'$ be the multidegree of $L|_{Y^c}(-Y \cap Y^c)$, for which we need to show that it is uniform on $H$.
    
    If a vertex $w$ of $H$ is not adjacent to $\Dh(v, \md)$ in $\GG_X$, then $\md_w = \md_w'$ and the valence of $w$ is the same in $\GG_X$ and $H$. Hence $\md'$ is uniform at $w$ since so is $\md$. 
    
    Now suppose $w$ is adjacent to $\Dh(v, \md)$ in $\GG_X$. Then $\md'_w = \md_w -k$, where $k = |X_w \cap Y|$ is the number of edges adjacent to both $w$ and $\Dh(v, \md)$. Observe first, that $\md'_w \geq 0$ since otherwise we would need to have $w \in \Dh(v, \md)$ by the construction of $\Dh(v, \md)$. Thus the lower bound for being uniform is satisfied. For the upper bound, observe that by assumption $\md_w \leq 2g_w - 2 + \val_{\GG_X}(w)$, where $\val_{\GG_X}(w)$ denotes the valence of $w$ in $\GG_X$. Since $\val_{\GG_X}(w) = \val_H(w) + k$, where $k$ is as above and $\val_H(w)$ denotes the valence of $w$ in $H$, we also get the upper bound \[\md'_w = \md_w - k \leq 2g_w - 2 + \val_{\GG_X}(w) - k = 2g_w - 2 + \val_{H}(w).\]
\end{proof}

\section{Counterexamples} \label{sec:counterexamples}

In this section, we construct examples of uniform multidegrees that do not satisfy the classic Clifford inequality $h^0(X,L) \leq \frac{d}{2} + 1$. In Proposition~\ref{prop:vertex of valence three}, we show that in fact on every semistable curve $X$ there exist line bundles of uniform multidegree, that achieve equality in Theorem~\ref{thm:clifford for uniform multidegrees}.

\begin{ex} \label{ex:basic}
Let $X$ be the stable curve consisting of three smooth, irreducible genus one curves $X_i$, each attached along a single node $p_i$ to a smooth, irreducible rational curve $X_v$. Consider the multidegree $\md$ with $\ud_v = 0$ and $\ud_{v_i} = 1$ where $v_i$ is the vertex corresponding to $X_i$. Thus $\md$ is uniform of total degree $3$. Let $L$ be the line bundle of multidegree $\md$ whose restriction to $X_i$ is $\mathcal O_{X_i}(p_i)$. In particular, $p_i$ is a base point of $L|_{X_i}$ and global sections of $L$ vanish along $X_v$. Thus by Lemma~\ref{lma:subcurves}, \[h^0(X,L) = h^0(X_v^c, L|_{X_v^c}\left(-X_v \cap X_v^c)\right) = \sum_{i =1}^3 h^0(X_i, \cO_{X_i}) = 3 > \frac{3}{2} + 1.\]
\end{ex}
\begin{rmk}\label{rmk:other statements clifford}
Example~\ref{ex:basic} contradicts the claim in \cite[Theorem 3.14]{Franciosi} that every line bundle of uniform multidegree satisfies the classic Clifford inequality. In case of Example~\ref{ex:basic}, Equation (9) in the proof of \emph{loc. cit.} no longer holds. The example also contradicts the claim of \cite[Proposition 3.1]{caporasosemistable}, that every multidegree $\md$ with $0 \leq \md_v \leq 2g_v$ satisfies the classic Clifford inequality. Here the issue is, that in an inductive argument, the claim is applied to $L(-p)$ for some smooth point $p$ of $X$. But $L(-p)$ can have negative degree on the irreducible component containing $p$, which then is outside the range of the induction. 
\end{rmk}

\begin{ex}
More generally, let $X$ be a stable curve and $\omega_X$ its dualizing sheaf. Suppose $\omega_X$ has a smooth base point $p \in X$. We then have 
\[
h^0(X, \omega_X(-p)) = h^0(X, \omega_X) = g > g - \frac{1}{2} = \frac{2g - 3}{2} + 1. 
\]
Thus $\omega_X(-p)$ does not satisfy the classic Clifford inequality. On the other hand, the multidegree of $\omega_X(-p)$ is uniform since $X$ is stable. By \cite[Theorem D, p. 75]{Catanese}, the dualizing sheaf $\omega_X$ has smooth base points precisely along smooth rational components $X_v$ such that all points in $X_v \cap X_v^c$ are separating nodes. In this case, the global sections of $\omega_X$ vanish on all of $X_v$. Thus we can subtract up to $\val(v) - 2$ smooth points on $X_v$ from $\omega_X$, obtaining in each case a line bundle with uniform multidegree not satisfying the classic Clifford inequality.   
\end{ex}

Recall that we defined the forest of $2$-edge-connected components $\GG_X^\Br$ of the dual graph $\GG_X$ in Definition~\ref{def:G2}. Furthermore, we denote by $\lea\left(\GG_X^\Br\right)$ its number of leaves as in Definition~\ref{def:number of leaves}. 

\begin{prop}\label{prop:vertex of valence three}
Let $X$ be a semistable curve. Then there is a line bundle $L$ on $X$ with uniform multidegree  such that
\[
h^0(X, L) = \frac{\deg(L)}{2} + \frac{\lea\left(\GG_X^\Br\right)}{2}.
\]

\end{prop}

\begin{proof}
    We may assume that $X$ is connected, since both sides of the equation are additive on connected components. If $\lea\left(\GG_X^\Br\right) = 2$, the claim is the classic Clifford inequality, and we may choose for $L$ either the dualizing sheaf or the structure sheaf of $X$. 
    
    Otherwise, set $l := \lea\left(\GG_X^\Br\right)$. Let $v_i$ with $1 \leq i \leq l$ denote the leaves of $\GG_X^\Br$. Each of them corresponds to a $2$-edge-connected subgraph of $\GG_X$, and we denote by $Y_i \subset X$ the corresponding subcurves. Let \[Y = \left( \bigcup_i Y_i \right)^c\] be the closure of the complement in $X$. It is not empty, since $\lea\left(\GG_X^\Br\right) \geq 3$ implies that there is a vertex of $\GG_X^\Br$ that is not a leaf. Set $p_i = Y_i \cap Y$, a collection of $l$ distinct separating nodes of $X$.
   
   Since the $Y_i$ are by definition connected to $Y$ along separating nodes, we have $\Pic(X) \simeq \Pic(Y) \times  \prod_i \Pic(Y_i)$ (see the discussion at the end of Section~\ref{sec:conventions}). Thus any line bundle on $X$ is up to automorphism specified by its restrictions to $Y$ and the $Y_i$. \emph{Vice versa}, any collection of line bundles on the $Y_i$ glues to a unique line bundle on $X$. 
   
   We define $L$ by setting $L|_Y = \mathcal O_Y$ and $L|_{Y_i} = \omega_{Y_i}(p_i)$ where $\omega_{Y_i}$ denotes the dualizing sheaf of $Y_i$. By construction, $L$ has uniform multidegree. By Corollary~\ref{cor:RR applied2}, $p_i$ is a base point of $L|_{Y_i}$. Thus global sections of $L$ vanish on all of $Y$ and by Lemma~\ref{lma:subcurves}:
   \[
   h^0(X, L) = \sum_{i = 1}^l h^0\left(Y_i, L|_{Y_i}(-p_i)\right) = \sum_{i = 1}^l h^0\left(Y_i, \omega_{Y_i}\right) = \sum_{i = 1}^l g(Y_i).
   \]
   On the other hand, we get for the total degree of $L$:
   \[
   \deg(L) = \sum_{i = 1}^l \left (2 g(Y_i) - 1\right )  = 2 \sum_{i = 1}^l g(Y_i) - l .
   \]
   And thus as claimed: 
   \[
h^0(X, L) = \sum_{i = 1}^l g(Y_i) = \frac{2 \sum_{i = 1}^l g(Y_i) - l}{2} + \frac{l}{2} = \frac{\deg(L)}{2} + \frac{\lea\left(\GG_X^\Br\right)}{2}.
\]
\end{proof}

\begin{rmk}
The construction in the proof of Proposition~\ref{prop:vertex of valence three} is not the only way to obtain a line bundle realizing the upper bound in the Clifford inequality for uniform multidegrees. For example, setting (some of the) $L|_{Y_i}$ to be $\mathcal O_{Y_i}(p_i)$ gives another way to construct $\md$ and $L$.
\end{rmk}

\section{The Inequality}

\subsection{The Clifford inequality for uniform multidegrees} \label{subsec:clifford inequality}
Recall that we defined the forest of $2$-edge connected components $\GG_X^\Br$ of the dual graph of $\GG_X$ in Definition~\ref{def:G2}. Recall furthermore, that $\lea\left(\GG_X^\Br\right)$ denotes it's number of leaves as in Definition~\ref{def:number of leaves}.

\begin{thm} \label{thm:clifford for uniform multidegrees}
Let $X$ be a semistable curve and $\ud$ a uniform multidegree of total degree $d$. Then every line bundle $L$ of multidegree $\md$ satisfies
\[
h^0(X, L) \leq \frac{d}{2} + \frac{\lea\left(\GG_X^\Br\right)}{2}.
\]
\end{thm}

Note, in particular, that $\lea\left(\GG_X^\Br\right) = 2$ if $X$ is connected and has no separating nodes, and hence every line bundle of uniform multidegree satisfies the classic Clifford inequality in this case. 

\begin{proof}
    We prove the claim by induction on the number of non-loop edges of $\GG_X$. The base of the induction is the case when $X$ is a disjoint union of irreducible curves. In this case, let $X_1, \ldots X_k$ denote the connected components of $X$ corresponding to vertices $v_i$ of $\GG_X$. Then by definition, $\lea \left(\GG_X^\Br\right) = 2k$. On the other hand, the restriction of $L$ to each connected component $X_i$ satisfies the classic Clifford inequality by Theorem~\ref{thm:clifford irreducible}, and hence we get \[h^0(X,L) = \sum_{i = 1}^k h^0\left(X_i, L|_{X_i}\right) \leq \sum_{i = 1}^k\left(\frac{\md_{v_i}}{2} + 1\right) = \frac{d}{2} + \frac{\lea \left(\GG_X^\Br \right)}{2}.\] 
    
    For the induction step, let $L$ be a line bundle with uniform multidegree $\md$ and total degree $d$ on $X$. Using the additivity on connected components of both sides of the claimed inequality as in the base of the induction, it suffices to show the claim for each connected component of $X$. So we assume $X$ to be connected from now on. We deal successively with different cases, that we then exclude going forward.
    
    \medskip
    
    \emph{Step 0: Suppose $X$ is not stable.}
      If $X$ has genus $1$, $L$ needs to have degree $0$ on each irreducible component and the claim holds by the Riemann-Roch Theorem. Otherwise $g \geq 2$ and $X$ contains a rational component $X_v$ with $\val(v) = 2$. Then the dual graph of the stabilization $X_\stab$ of $X$ has less non-loop edges than that of $X$. Furthermore, we have  $\lea \left(\GG_X^\Br \right)  = \lea \left(\GG_{X_\stab}^\Br \right)$ by Lemma~\ref{lma:leaves_stabilization}. Hence the claim follows for $X$ by induction and Corollary~\ref{cor:cliffordindex}. 
    
    \medskip
    
    So we assume from now on, that $X$ is stable.
    
    \medskip
    
    \emph{Step 1: Suppose $\md_v = 0$ for some vertex $v$ and all global sections of $L$ vanish on all of $X_v$.}
    Let $\Dh(v,\md)$ be the Dhar subgraph associated to $v$ and $\md$, as in Subsection~\ref{subsec:Dhar}. Denote by $H$ the induced subgraph containing all vertices of $\GG_X$ not contained in $\Dh(v, \md)$. Let $Y \subset X$ be the subcurve corresponding to $\Dh(v, \md)$, and $Y^c$ the closure of its complement, which is the subcurve corresponding to $H$. Set $k = |Y \cap Y^c|$. By the assumption of Step 1, all global sections of $L$ vanish on all of $X_v$ and hence by Lemma~\ref{lma:dhar restriction} also on all of $Y$. Thus if $Y^c$ is empty, there is nothing to show. Otherwise, we get by Lemma~\ref{lma:subcurves} that \begin{equation} \label{eq:restriction to Yc1}
        h^0(X,L) = h^0\left(Y^c, L|_{Y^c}(- Y \cap Y^c)\right).
    \end{equation}
    By Lemma~\ref{lma:dhar uniform}, $Y^c$ is semistable and the multidegree of $L|_{Y^c}(- Y \cap Y^c)$ is uniform on $Y^c$. Furthermore, $H$ contains less edges than $\GG_X$, and hence we obtain by induction \begin{equation} \label{eq:restriction to Yc2}
        h^0\left(Y^c, L|_{Y^c}(- Y \cap Y^c)\right) \leq \frac{d - k}{2} + \frac{\lea \left(H^\Br \right)}{2}.
    \end{equation}
    Now Lemma~\ref{lma:splitting of number of leaves}~(2) gives $\lea \left(H^\Br \right) \leq \lea \left(\GG_X^\Br \right) + k$, and using \eqref{eq:restriction to Yc2} and \eqref{eq:restriction to Yc1} we get as claimed
    \[h^0(X,L) = h^0\left(Y^c, L|_{Y^c}(- Y \cap Y^c)\right) \leq \frac{d - k}{2} + \frac{\lea \left(H^\Br \right)}{2} \leq \frac{d}{2} + \frac{\lea \left(\GG_X^\Br \right)}{2}.\]
    
    \medskip
    
    So we assume from now on, that whenever $\md_v = 0$, not all global sections of $L$ vanish on all of $X_v$. We may assume the same for the residual $\omega_X \otimes L^{-1}$: by Lemma~\ref{cor: clifford residual} $L$ satisfies the claim if and only if $\omega_X \otimes L^{-1}$ does, and $\omega_X \otimes L^{-1}$ has uniform multidegree since $L$ does, see Remark~\ref{rmk:uniform residual}.
    
    \medskip

    \emph{Step 2: Suppose $\GG_X$ contains a bridge $e$.}
    Let $p$ be the separating node of $X$ corresponding to $e$. Let $X_1, X_2$ be the two connected components of the partial normalization of $X$ at $p$, and $L_1, L_2$ the respective pull backs of $L$. Let $p_1 \in X_1, p_2 \in X_2$ be the preimages of $p$ and denote by $X_v, X_w$ the irreducible components containing $p_1$ and $p_2$, respectively. Let $G_1, G_2$ denote the subgraphs corresponding to the $X_i$, that is, the connected components of $\GG_X - e$. Since we assume $X$ to be stable following Step 0, the $X_i$ are both semistable. By Lemma~\ref{lma:splitting of number of leaves} (1) and since the $G_i$ are connected, we have \begin{equation}\label{eq:step2.1}
        \lea \left(G_1^\Br \right) + \lea \left(G_2^\Br \right) \leq \lea \left(\GG_X^\Br \right) + 2.
    \end{equation} On the other hand, Lemma~\ref{lma:neutral pairs} gives \begin{equation}\label{eq:step2.2}
        h^0(X,L) \leq h^0(X_1, L_1) + h^0(X_2, L_2) - c,
    \end{equation} 
    where $c = 0$ if $p_1$ is a base point of $L_1$ and $p_2$ is a base point of $L_2$ and $c = 1$ otherwise.
    
    Suppose first that $c = 0$ and thus that $p_i$ is a base point of $L_i$ and similarly all global sections of $L$ vanish at $p$. We claim that if $\md_v = 0$, then all global sections of $L$ need to vanish along $X_v$, which we excluded after Step 1. Indeed, consider the restriction map $H^0(X, L) \to H^0(X_v, L|_{X_v})$. If $L|_{X_v}$ has no global sections, then all global sections of $L$ vanish along $X_v$. Otherwise we need to have $L|_{X_v} \simeq \cO_{X_v}$ since $\deg(L|_{X_v}) = 0$. In particular, $h^0(X_v, L|_{X_v}) = 1$ and $L|_{X_v}$ is base point free. Since $L_1$ does have a base point on $X_v$, it follows that the restriction map $H^0(X, L) \to H^0(X_v, L|_{X_v})$ is not surjective and hence the zero map. Thus also in this case all global sections of $L$ vanish along $X_v$. So we may assume $\md_v > 0$ and, by the same argument for $X_w$, $\md_w > 0$. 
    
    With this assumption $L_1(-p_1)$ and $L_2(-p_2)$ have uniform multidegree on $X_1$ and $X_2$, respectively. By induction and since the $p_i$ are base points we get: \[h^0\left(X_i, L_i\right) = h^0\left(X_i,L_i(-p_i)\right) \leq \frac{\deg(L_i) - 1}{2} + \frac{\lea\left(G_i^\Br\right)}{2},\]
    Using this together with \eqref{eq:step2.2} and \eqref{eq:step2.1} then gives as claimed
    \[ h^0(X,L) \leq h^0(X_1, L_1) + h^0(X_2, L_2) \leq \frac{\deg(L_1) - 1}{2} + \frac{\lea\left(G_1^\Br\right)}{2} + \frac{\deg(L_2) - 1}{2} + \frac{\lea\left(G_2^\Br\right)}{2} \leq \frac{d}{2} + \frac{\lea \left(\GG_X^\Br \right)}{2}\]
    
    Suppose next, that $c = 1$. In this case, it suffices to show that \begin{equation}\label{eq:step2.3}
        h^0(X_i, L_i) \leq \frac{d_i}{2} + \frac{\lea\left(G_i^\Br\right)}{2},
    \end{equation} which inserted in \eqref{eq:step2.2} gives the claim together with \eqref{eq:step2.1}. To show \eqref{eq:step2.3}, observe that it follows immediately by induction if $L_1$ and $L_2$ have uniform multidegree on $X_1$ and $X_2$. Otherwise, if, say, $L_1$ is not uniform, we need to have $\md_v = 2g_v - 2 + \val(v)$, since the valence of $v$ decreases by only one in passing from $\GG_X$ to $G_1$. In particular, $L_1(-p_1)$ has uniform multidegree on $X_1$. Consider the residual $\omega_X \otimes L^{-1}$. It has degree $0$ on $X_v$ and by the assumption following Step 1, not all global sections of $\omega_X \otimes L^{-1}$ vanish on $X_v$. Hence it has no base points on $X_v$. Observe that \[\left(\omega_X \otimes L^{-1}\right)|_{X_1} \simeq \left(\omega_{X_1} \otimes L_1^{-1}\right)(p).\] Thus since $p$ is not a base point of $\left(\omega_X \otimes L^{-1}\right)|_{X_1}$, it is also not a base point of $\left(\omega_{X_1} \otimes L_1^{-1}\right)(p)$. By Corollary~\ref{cor:RR applied1} this implies that $p$ is a base point of $L_1$, and hence $h^0(X_1, L_1) = h^0\left(X_1, L_1(-p)\right)$. Since $L_1(-p)$ has uniform multidegree, \eqref{eq:step2.3} follows by induction for $L_1$. Analogously for $L_2$.
    
    \medskip
    
    So we assume from now on, that $\GG_X$ contains no bridges, that is, it is $2$-edge-connected. In this case, we need to show that $L$ satisfies the classic Clifford inequality $h^0(X,L) \leq \frac{d}{2} + 1$.
    
    \medskip
    
    \emph{Step 3: Suppose $\md_v = 0$ for some vertex $v$ of $\GG_X$.}
    We argue similarly as in Step 1, but now may employ the stronger version Lemma~\ref{lma:splitting of number of leaves} (2) instead of Lemma~\ref{lma:splitting of number of leaves} (1). As before, we denote by $\Dh(v, \md)$ the Dhar set of $v$ and $\md$, by $Y$ the subcurve corresponding to $\Dh(v, \md)$ and by $Y^c$ the closure of its complement corresponding to the induced subgraph $H$. Set $k = |Y \cap Y^c|$. Then by Lemma~\ref{lma:splitting of number of leaves} (2), we have 
    \begin{equation}\label{eq:Step3.1}
        \lea\left(H^\Br\right) \leq k.
    \end{equation} By Lemma~\ref{lma:dhar restriction} we have that $h^0(Y, L|_Y) = h^0(X_v, L|_{X_v}) = 1$ and thus we get by Lemma~\ref{lma:subcurves} that \begin{equation}\label{eq:Step3.2}
        h^0(X,L) \leq 1  + h^0\left(Y^c, L|_{Y^c}(-Y \cap Y^c)\right).
    \end{equation}
    By Lemma~\ref{lma:dhar uniform}, $Y^c$ is semistable and $L|_{Y^c}(-Y \cap Y^c)$ has uniform multidegree on $Y^c$ and hence by induction and \eqref{eq:Step3.1} \begin{equation}\label{eq:Step3.3}
        h^0(Y^c, L|_{Y^c}(-Y \cap Y^c) \leq \frac{d - k}{2} + \frac{\lea\left(H^\Br\right)}{2} \leq \frac{d}{2}.
    \end{equation}
    The claim in this case follows by combining \eqref{eq:Step3.3} and \eqref{eq:Step3.2}.
    \medskip
    
    So we assume from now on, that $\md_v > 0$ for all vertices $v$ of $\GG_X$.
    
    \medskip
    
    \emph{Step 4: Conclusion.}
    Let $e$ be a non-loop edge of $\GG_X$ corresponding to a node $p$ of $X$. Let $X^\nu$ be the partial normalization of $X$ at $p$ and $L^\nu$ the pull back of $L$ to $X^\nu$. Since we assume that $X$ is stable and contains no separating nodes, $X^\nu$ is semistable and connected. Denote by $p_1, p_2$ the two preimages of $p$ in $X^\nu$. Then we get by Lemma~\ref{lma:neutral pairs}: \begin{equation}\label{eq:Step4.1}
        h^0(X,L) = h^0\left(X^\nu,L^\nu\right) - c,
    \end{equation} where $c = 1$ if $p_1, p_2$ are not a neutral pair of $L^\nu$ and $c = 0$ otherwise (see Definition~\ref{def: neutral pair} for the definition of neutral pairs). Furthermore we have by definition that 
    \begin{equation}\label{eq:Step4.2}
     h^0\left(X^\nu,L^\nu\right) = h^0\left(X^\nu,L^\nu(- p_1 - p_2)\right) + c',
    \end{equation}
    where $c' \leq 1$ if the $p_i$ are a neutral pair of $L^\nu$, and $c' = 2$ otherwise. Thus in any case $c' - c \leq 1$. Since $\md_v > 0$ for all $v$ by assumption and $e$ is not a loop edge, we have that $L^\nu(-p_1 - p_2)$ has uniform multidegree on $X^\nu$. Furthermore, since $\GG_X$ is $2$-edge-connected,  Lemma~\ref{lma: removing edge} gives $\lea\left((\GG_X - e)^\Br\right) = 2$. Since $\GG_X - e$ is the dual graph of $X^\nu$, we get by combining \eqref{eq:Step4.1},  \eqref{eq:Step4.2} and the induction assumption as claimed: \[h^0(X,L) \leq h^0\left(X^\nu,L^\nu\right) - c \leq h^0\left(X^\nu,L^\nu(- p_1 - p_2)\right) + c' - c \leq \frac{d - 2}{2} + 1 + 1.\]
\end{proof}

\subsection{Generic behaviour} \label{subsec: generic}
 Recall that we denote by $\Pic^{\md}(X) \subset \Pic(X)$ the irreducible component of the Picard scheme, that parametrizes line bundles of multidegree $\md$.

\begin{prop}\label{prop:generic case}
Let $X$ be a connected semistable curve and $\md$ a uniform multidegree on $X$. Then there is a dense open subset $U$ of $\Pic^{\md}(X)$ such that every $L$ contained in $U$ satisfies the classic Clifford inequality \[h^0(X,L) \leq \frac{d}{2} + 1.\]
\end{prop}

\begin{proof}
    We prove the claim by induction on the number of non-loop edges in $\GG_X$. As base case we use the case $\lea\left(\GG_X^\Br\right) = 2$, in which case Theorem~\ref{thm:clifford for uniform multidegrees} gives the claim for any line bundle of uniform multidegree. 
    
    So suppose $\lea\left(\GG_X^\Br\right) \geq 3$. We choose a non-loop edge $e$ of $\GG_X$ and consider the partial normalization $X^\nu$ of $X$ at the node $p$ corresponding to $e$. Denote by $p_1, p_2$ the two preimages of $p$ in $X^\nu$ and by $L^\nu$ the pull back of a line bundle $L$ of multidegree $\md$.
    
    Arguing as in Step 0 in the proof of Theorem~\ref{thm:clifford for uniform multidegrees}, we may assume that $X$ is stable.
    \medskip
    
    \emph{Step 1: Suppose every non-loop edge of $\GG_X$ is a bridge.} We may assume that $e$ is the unique non-loop edge adjacent to a vertex $v$ of $\GG_X$. Then $X^\nu$ has two connected components, $X_v$ and $X'$. Since $X$ is stable, $X'$ is semistable. Since either $L|_{X'}$ or $L|_{X'}(-p_1)$ has uniform multidegree on $X'$, we get by induction for $L|_{X'}$ general
    \begin{equation}\label{eq:step1.1'}
        h^0\left(X', L|_{X'}\right) \leq \frac{\deg(L|_{X'})}{2} + c,
    \end{equation}
    with $c = 1$ if $p_1$ is a base point of $L|_{X'}$ and $c = \frac{3}{2}$ otherwise. 
    
    On the other hand, $\md_v \leq 2g_v - 1$, thus either by the Clifford inequality~\ref{thm:clifford irreducible} for irreducible curves (if $\md_v \leq 2g_v - 2$) or by the Riemann-Roch Theorem~\ref{thm:Riemann Roch} (if $\md_v = 2g_v - 1$), we get for $L|_{X_v}$ general 
    \begin{equation}\label{eq:step1.2'}
    h^0\left(X_v, L|_{X_v}\right) \leq \frac{\md_v}{2} + \frac{1}{2}.
    \end{equation}
    
    Suppose first $h^0\left(X_v, L|_{X_v}\right) = 0$. Then $h^0(X,L) = h^0\left(X', L|_{X'}\right) - c'$ by Lemma~\ref{lma:neutral pairs}, with $c' = 0$ if $p_1$ is a base point of $L|_{X'}$ and $c' = 1$ otherwise. Thus the claim follows from \eqref{eq:step1.1'}. 
    
    Now suppose $h^0\left(X_v, L|_{X_v}\right) > 0$. Then for $L|_{X_v}$ general, $p_2$ is not a base point of $L|_{X_v}$.
    Thus we get by Lemma~\ref{lma:neutral pairs}: 
    \begin{equation*}\label{eq:step1.3'}
    h^0(X,L) \leq h^0\left(X', L'|_{X'}\right) + h^0\left(X_v, L|_{X_v}\right) - 1.
    \end{equation*}
    Inserting \eqref{eq:step1.1'} and \eqref{eq:step1.2'} this gives as claimed 
    \[ h^0(X,L) \leq h^0\left(X', L'|_{X'}\right) + h^0\left(X_v, L|_{X_v}\right) - 1 \leq \frac{\deg(L|_{X'})}{2} + \frac{3}{2} + \frac{\md_v}{2} + \frac{1}{2} - 1 = \frac{d}{2} + 1. \]

    \medskip

   \emph{Step 2: Conclusion.} It remains to show the claim if $e$ is neither a bridge nor a loop. Since $X$ is stable and $e$ is not a loop, $X^\nu$ is semistable. We distinguish two cases.
   
       Case 1: both $p_1$ and $p_2$ are base points of $L^\nu$. Then by Lemma~\ref{lma:neutral pairs} we have \[h^0(X,L) = h^0\left(X^\nu, L^{\nu}\right) = h^0\left(X^\nu, L^{\nu}(-p_i)\right) = h^0\left(X^\nu, L^{\nu}(-p_1 - p_2)\right).\] Since at least one of $L^{\nu}, L^\nu(-p_1), L^\nu(-p_2)$ and $L^\nu(-p_1 - p_2)$ has uniform multidegree on $X^\nu$, the claim follows by induction for $L$ general.

       Case 2: Otherwise, and again by Lemma~\ref{lma:neutral pairs}, we have \begin{equation}\label{eq:Step2.1'}
           h^0(X,L) = h^0(X^\nu, L^\nu) - 1,
       \end{equation} if $L$ is general. Indeed, if $p_1$ and $p_2$ are not a neutral pair $L^\nu$, then \eqref{eq:Step2.1'} holds for every $L$ that pulls back to $L^\nu$. If on the other hand $p_1$ and $p_2$ are a neutral pair of $L^\nu$, they are by assumption not base points. Then the last claim of Lemma~\ref{lma:neutral pairs} ensures that \eqref{eq:Step2.1'} holds for all but one $L$ that pulls back to $L^\nu$. Since $e$ is not a bridge, the pull-back map $\Pic (X) \to \Pic (X^\nu)$ has fiber $k^*$ and hence \eqref{eq:Step2.1'} holds for a general $L$. We again have that at least one of $L^{\nu}, L^\nu(-p_1), L^\nu(-p_2)$ and $L^\nu(-p_1 - p_2)$ is uniform, and hence induction gives 
       \begin{equation}\label{eq:Step2.2'}
       h^0(X^\nu,L^\nu) \leq \frac{d}{2} + 2.
       \end{equation}
   Inserting \eqref{eq:Step2.2'} in \eqref{eq:Step2.1'} then gives the claim.
\end{proof}

\subsection{Relation to stable multidegrees} \label{subsec:semistable multidegrees}

Finally, we discuss another important class of multidegrees, the stable ones. They are essential in the construction of universal compactified Jacobians \cite{Caporasocompactification}.
We restrict to stable multidegrees of total degrees $d = g - 1$ or $d =  g$, since all phenomena already appear and the definition of stability is significantly easier. We refer the interested reader to \cite[\S 1.3]{caporasotheta}, \cite[\S\S 3.2 and 3.3]{CC}, \cite[\S 5.4]{CPS} or \cite[\S\S 4 and 5]{chr1} for details about semistability and compactified Jacobians in these two cases.

\begin{defn} \label{def:stable}
   Let $X$ be a stable curve and $\md$ a multidegree of total degree $d \in \left \{g-1, g \right\}$. Then $\md$ is called \emph{stable}, if for every proper subcurve $Y \subset X$ of genus $g(Y)$ we have \[\sum_{X_v \subset Y} \md_v \geq g(Y).\]
\end{defn}

Stable multidegrees of total degree $g-1$ exist if and only if $X$ contains no separating nodes. Stable multidegrees of total degree $g$ exist on any $X$. 

In general, stable multidegrees are not uniform and \emph{vice versa}. In particular, as the following example shows, stable multidegrees need not satisfy the Clifford inequality for uniform multidegrees established in Theorem~\ref{thm:clifford for uniform multidegrees} (see also \cite[Example 4.15 and 4.17]{caporasosemistable}). 

\begin{ex}\label{ex:semistable clifford}
 	Let $\GG_X$ and $\md$ be as in Figure~\ref{figure4}. Let $v_1$ be the $9$-valent vertex in the middle and $v_i$, $2 \leq i \leq 10$, the $2$-valent vertices of weight $1$. Denote by $X_i$ the irreducible components of $X$ corresponding to the vertices $v_i$ and let $p_i = X_1 \cap X_i$. Let $L$ have multidegree $\md$ and satisfy $L|_{X_i} = \mathcal O_{X_i}(p_i)$ for $i \geq 2$ and an arbitrary choice of gluing data. One checks that \[h^0(X,L) = 9 > 8,5 = \frac{d}{2} + 1.\] But $X$ is stable without separating nodes and $\md$ is a stable multidegree. Hence the Clifford inequality for uniform multidegrees of Theorem~\ref{thm:clifford for uniform multidegrees} would be $h^0(X,L) \leq \frac{d}{2} + 1$ and is not satisfied.
\end{ex}

\tikzset{every picture/.style={line width=0.75pt}}	

\begin{figure}[ht]
    \begin{tikzpicture}[x=0.75pt,y=0.75pt,yscale=-0.75,xscale=0.75]
\import{./}{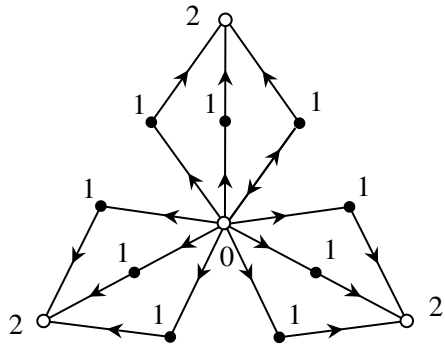}
\end{tikzpicture}	
    \caption{
       The graph $\GG_X$ and the multidegree $\md$ of Example~\ref{ex:semistable clifford}. Bold vertices have weight $1$ and vertices drawn as circles weight $0$. We included the generalized orientation giving $\md$, and stability of $\md$ follows from \cite[Lemma 3.3.2]{CC}.
    } \label{figure4}
\end{figure}
 	
A \emph{general} line bundle in $\Pic^{\md}(X)$ with $\md$ stable satisfies the Clifford inequality. In fact, \cite{chr1} establishes the much stronger claim that a general semistable line bundle $L$ is non-special, i.e., $h^0(X,L) = \max \left\{0, d - g + 1 \right\}$, which in particular implies the classic Clifford inequality for general line bundles if $0 \leq d \leq 2g - 2$.
 
 There are two special cases, for which the bound on \emph{all} line bundles in $\Pic^{\md}(X)$ established in Theorem~\ref{thm:clifford for uniform multidegrees} applies also to stable multidegrees $\md$:
 	
 	\begin{lma} \label{lma:stable in degree g-1}
 	Let $X$ be a stable curve and $\md$ a stable multidegree of total degree $d$. Then $\md$ is uniform if either $d = g - 1$, or $d = g$ and $X$ contains no irreducible components of geometric genus $0$.
 	\end{lma}
 	\begin{proof}
    Definition~\ref{def:stable} always gives $\md_v \geq 0$ for all vertices $v$. It remains to show $\md_v \leq 2g_v - 2 +\val(v)$.
 	To see this, we apply Definition~\ref{def:stable} to $X_v^c$, the closure of the complement of $X_v$ in $X$. We get \[d - \md_v \geq g(X^c_v).\]
 	Since $g(X^c_v) = g - g_v - \val(v) + 1$ we obtain \[\md_v \leq d - g(X^c_v) = d - g + g_v + \val(v) - 1.\]
 	Now if either $d = g - 1$ or $d=g$ and $g_v \geq 1$, the claim follows.
 	\end{proof}
 	
 	\begin{rmk}\label{rmk: compactified Jacobians}
 	Suppose $X$ contains no separating nodes. Then the compactified Jacobian $\overline P_X^{g -1}$ in degree $g - 1$ parametrizes line bundles with stable multidegree of total degree $g - 1$ away from the boundary. By  Theorem~\ref{thm:clifford for uniform multidegrees} and Lemma~\ref{lma:stable in degree g-1}, all such line bundles satisfy the classic Clifford inequality. The boundary of $\overline P_X^{g -1}$ on the other hand parametrizes torsion-free rank $1$ sheaves, that satisfy a related stability condition. They need not satisfy the classic Clifford inequality. For example, suppose $X$ contains three irreducible components, each smooth and of genus $1$, with any two of them intersecting in a single node. Let $F$ be the sheaf that modulo torsion restricts to the structure sheaf on each of the $3$ irreducible components and fails to be locally free at all three nodes. Then $\deg(F) = 3$ and $h^0(X,F) = 3 > \frac{3}{2} + 1$ even though $X$ contains no separating nodes. 
 	\end{rmk}
 	
\providecommand{\bysame}{\leavevmode\hbox to3em{\hrulefill}\thinspace}
\providecommand{\MR}{\relax\ifhmode\unskip\space\fi MR }
\providecommand{\MRhref}[2]{%
  \href{http://www.ams.org/mathscinet-getitem?mr=#1}{#2}
}
\providecommand{\href}[2]{#2}


\begin{thebibliography}{CFHR99}

\bibitem[ACG11]{curvesv2}
E.~Arbarello, M.~Cornalba, and P.~A. Griffiths, \emph{Geometry of algebraic
  curves. {V}olume {II}}, Grundlehren der mathematischen Wissenschaften
  [Fundamental Principles of Mathematical Sciences], vol. 268, Springer,
  Heidelberg, 2011, With a contribution by J. Harris. \MR{2807457}

\bibitem[AE21]{AMIII}
O.~Amini and E.~Esteves, \emph{Voronoi tilings, toric arrangements and
  degenerations of line bundles {III}}, 2021, arXiv: 2102.02149.

\bibitem[Ale04]{Alexeev}
V.~Alexeev, \emph{Compactified {J}acobians and {T}orelli map}, Publ. Res. Inst.
  Math. Sci. \textbf{40} (2004), no.~4, 1241--1265. \MR{2105707}

\bibitem[Bak08]{Baker}
M.~Baker, \emph{Specialization of linear systems from curves to graphs},
  Algebra Number Theory \textbf{2} (2008), no.~6, 613--653, With an appendix by
  Brian Conrad. \MR{2448666}

\bibitem[BE91]{BE91}
D.~Bayer and D.~Eisenbud, \emph{Graph curves}, Adv. Math. \textbf{86} (1991),
  no.~1, 1--40, With an appendix by Sung Won Park. \MR{1097026}

\bibitem[BN07]{BN}
M.~Baker and S.~Norine, \emph{Riemann-{R}och and {A}bel-{J}acobi theory on a
  finite graph}, Adv. Math. \textbf{215} (2007), no.~2, 766--788. \MR{2355607}

\bibitem[Cap94]{Caporasocompactification}
L.~Caporaso, \emph{A compactification of the universal {P}icard variety over
  the moduli space of stable curves}, J. Amer. Math. Soc. \textbf{7} (1994),
  no.~3, 589--660. \MR{1254134}

\bibitem[Cap09]{caporasotheta}
\bysame, \emph{Geometry of the theta divisor of a compactified {J}acobian}, J.
  Eur. Math. Soc. (JEMS) \textbf{11} (2009), no.~6, 1385--1427. \MR{2557139}

\bibitem[Cap11]{caporasosemistable}
\bysame, \emph{Linear series on semistable curves}, Int. Math. Res. Not. IMRN
  (2011), no.~13, 2921--2969. \MR{2817683}

\bibitem[Cat82]{Catanese}
F.~Catanese, \emph{Pluricanonical-{G}orenstein-curves}, Enumerative geometry
  and classical algebraic geometry ({N}ice, 1981), Progr. Math., vol.~24,
  Birkh\"{a}user Boston, Boston, MA, 1982, pp.~51--95. \MR{685764}

\bibitem[CC19]{CC}
L.~Caporaso and K.~Christ, \emph{Combinatorics of compactified universal
  {J}acobians}, Adv. Math. \textbf{346} (2019), 1091--1136. \MR{3914907}

\bibitem[CFHR99]{CFHR}
F.~Catanese, M.~Franciosi, K.~Hulek, and M.~Reid, \emph{Embeddings of curves
  and surfaces}, Nagoya Math. J. \textbf{154} (1999), 185--220. \MR{1689180}

\bibitem[Chr20]{chr1}
K.~Christ, \emph{On the rank of general linear series on stable curves}, 2020,
  arXiv: 2005.12817.

\bibitem[CLM15]{CLM}
L.~Caporaso, Y.~Len, and M.~Melo, \emph{Algebraic and combinatorial rank of
  divisors on finite graphs}, J. Math. Pures Appl. (9) \textbf{104} (2015),
  no.~2, 227--257. \MR{3365828}

\bibitem[CPS19]{CPS}
K.~Christ, S.~Payne, and T.~Shen, \emph{Compactified {J}acobians as {M}umford
  models}, Transactions of the AMS (to appear), arXiv: 1912.03653v1, 2019.

\bibitem[DM69]{DM}
P.~Deligne and D.~Mumford, \emph{The irreducibility of the space of curves of
  given genus}, Inst. Hautes \'{E}tudes Sci. Publ. Math. (1969), no.~36,
  75--109. \MR{262240}

\bibitem[EH86]{EH86}
D.~Eisenbud and J.~Harris, \emph{Limit linear series: basic theory}, Invent.
  Math. \textbf{85} (1986), no.~2, 337--371. \MR{846932}

\bibitem[EKS88]{EKS}
D.~Eisenbud, J.~Koh, and M.~Stillman, \emph{Determinantal equations for curves
  of high degree}, Amer. J. Math. \textbf{110} (1988), no.~3, 513--539, with an
  appendix with J. Harris. \MR{944326}

\bibitem[EP16]{EstevesPacini16}
E.~Esteves and M.~Pacini, \emph{Semistable modifications of families of curves
  and compactified {J}acobians}, Ark. Mat. \textbf{54} (2016), no.~1, 55--83.

\bibitem[Fra19]{Franciosi}
M.~Franciosi, \emph{Clifford index for reduced curves}, Ann. Mat. Pura Appl.
  (4) \textbf{198} (2019), no.~6, 2167--2181. \MR{4031845}

\bibitem[FT14]{FT14}
M.~Franciosi and E.~Tenni, \emph{On {C}lifford's theorem for singular curves},
  Proc. Lond. Math. Soc. (3) \textbf{108} (2014), no.~1, 225--252. \MR{3162826}

\bibitem[Mai98]{M98}
L.~Maino, \emph{Moduli space of enriched stable curves}, ProQuest LLC, Ann
  Arbor, MI, 1998, Thesis (Ph.D.)--Harvard University. \MR{2697466}

\bibitem[MW22]{MW22}
S.~Molcho and J.~Wise, \emph{The logarithmic {P}icard group and its
  tropicalization}, Compos. Math. \textbf{158} (2022), no.~7, 1477--1562.
  \MR{4479817}

\bibitem[OS79]{OdaSeshadri79}
T.~Oda and C.~S. Seshadri, \emph{Compactifications of the generalized
  {J}acobian variety}, Trans. Amer. Math. Soc. \textbf{253} (1979), 1--90.

\bibitem[Oss19]{O19}
B.~Osserman, \emph{Limit linear series for curves not of compact type}, J.
  Reine Angew. Math. \textbf{753} (2019), 57--88. \MR{3987864}

\end{thebibliography}
\end{document}